\documentclass[english,12pt,a4paper]{smfart}

\usepackage[T1]{fontenc}
\usepackage{lmodern}
\usepackage{smfthm}
\usepackage[headings]{fullpage}
\usepackage{amssymb}

\newcommand{\Qp}{\mathbf{Q}_p}
\newcommand{\Cp}{\mathbf{C}_p}
\newcommand{\Zp}{\mathbf{Z}_p}
\newcommand{\ZZ}{\mathbf{Z}}
\newcommand{\QQ}{\mathbf{Q}}
\newcommand{\OO}{\mathcal{O}}
\newcommand{\MM}{\mathfrak{m}}
\newcommand{\Qpbar}{\overline{\mathbf{Q}}_p}
\renewcommand{\phi}{\varphi}
\renewcommand{\projlim}{\varprojlim}
\renewcommand{\geq}{\geqslant}
 
\newcommand{\Nphi}{\mathcal{N}}
\newcommand{\Nm}{\mathrm{N}}
\newcommand{\bigO}{\mathrm{O}}
\newcommand{\Gal}{\mathrm{Gal}}
\newcommand{\Emb}{\mathrm{Emb}}
\newcommand{\Card}{\mathrm{Card}}
\newcommand{\Fil}{\mathrm{Fil}}
\newcommand{\GL}{\mathrm{GL}}
\newcommand{\sep}{\mathrm{sep}}
\newcommand{\val}{\mathrm{val}}
\newcommand{\vp}{\mathrm{val}_p}

\newcommand{\ac}{\mathrm{ac}}
\newcommand{\ab}{\mathrm{ab}}

\newcommand{\cyc}{\mathrm{cyc}}
\newcommand{\afont}{\mathbf{A}}
\newcommand{\efont}{\mathbf{E}}
\newcommand{\at}{\tilde{\mathbf{A}}}
\newcommand{\et}{\widetilde{\mathbf{E}}}
\newcommand{\atplus}{\tilde{\mathbf{A}}^+}
\newcommand{\etplus}{\widetilde{\mathbf{E}}^+}

\newcommand{\bdr}{\mathbf{B}_{\mathrm{dR}}}  
\newcommand{\dfont}{\mathrm{D}}
\newcommand{\dpar}[1]{(\!( #1 )\!)}
\newcommand{\dcroc}[1]{[\![ #1 ]\!]}
\newcommand{\cF}{\mathcal{F}}

\author{Laurent Berger}
\address{UMPA de l'ENS de Lyon \\
UMR 5669 du CNRS \\ IUF}
\email{laurent.berger@ens-lyon.fr}
\urladdr{perso.ens-lyon.fr/laurent.berger/}

\date{\today}

\title{Lifting the field of norms}



\begin{document}

\begin{abstract}
Let $K$ be a finite extension of $\mathbf{Q}_p$. The field of norms of a $p$-adic Lie extension $K_\infty/K$ is a local field of characteristic $p$ which comes equipped with an action of $\mathrm{Gal}(K_\infty/K)$. When can we lift this action to characteristic $0$, along with a compatible Frobenius map? In this note, we formulate precisely this question, explain its relevance to the theory of $(\varphi,\Gamma)$-modules, and give a condition for the existence of certain types of lifts.
\end{abstract}

\maketitle

\tableofcontents

\setlength{\baselineskip}{18pt}

\section*{Introduction}\label{intro} 

Let $K$ be a finite extension of $\Qp$ and let $K_\infty/K$ be a totally ramified Galois extension whose Galois group $\Gamma_K$ is a $p$-adic Lie group (or, more generally, a ``strictly arithmetically profinite'' extension). Let $k_K$ denote the residue field of $K$. We can attach to $K_\infty/K$ its field of norms $X_K(K_\infty)$, a field of characteristic $p$ that is isomorphic to $k_K \dpar{\pi}$ and equipped with an action of $\Gamma_K$. Let $E$ be a finite extension of $\Qp$ such that $k_E=k_K$. In this note, we consider the question: when can we lift the action of $\Gamma_K$ on $k_K \dpar{\pi}$ to the $p$-adic completion of $\OO_E \dcroc{T}[1/T]$, which is a complete ring of characteristic $0$ that lifts $X_K(K_\infty)$, along with a compatible $\OO_E$-linear Frobenius map $\phi_q$? When it is possible to do so, we say that the action of $\Gamma_K$ is \emph{liftable}. In this case, Fontaine's construction of $(\phi,\Gamma)$-modules applies, and we get the following well-known equivalence of categories (where $\afont_K$ denotes the $p$-adic completion of $\OO_E \dcroc{T}[1/T]$).

\begin{enonce*}{Theorem A}\label{pgmintro}
If the action of $\Gamma_K$ is liftable, then there is an equivalence of categories
\[ \text{$\{ (\phi_q,\Gamma_K)$-modules on $\afont_K \} \longleftrightarrow \{ \OO_E$-linear representations of $G_K\}$}. \] 
\end{enonce*}

Such a lift is possible when $K_\infty/K$ is the cyclotomic extension, or more generally when $K_\infty$ is generated by the torsion points of a Lubin-Tate formal $\OO_F$-module for some $F \subset K$. In \S \ref{finht} of this note, we prove the following partial converse.

\begin{enonce*}{Theorem B}\label{liftintro}
If the action of $\Gamma_K$ is liftable with $\phi_q(T) \in \OO_E \dcroc{T}$, then $\Gamma_K$ is abelian, and there is an injective character $\Gamma_K \to \OO_E^\times$, whose conjugates by $\Emb(E,\Qpbar)$ are all de Rham with weights in $\ZZ_{\geq 0}$.
\end{enonce*}

At the end of \S \ref{finht}, we give some examples of constraints on the extension $K_\infty/K$ arising from the existence of such a character.

Some preliminary computations suggest that a similar result may hold in certain cases if we assume that $\phi_q(T)$ is an overconvergent power series in $T$. However at this point, I do not know for which extensions we can expect the action of $\Gamma_K$ to be liftable in general.

The initial motivation for thinking about this problem was the question of whether there is a theory of ``anticyclotomic $(\phi,\Gamma)$-modules'', that is a theory of $(\phi,\Gamma)$-modules where $\Gamma$ is the Galois group of the anticyclotomic extension $K_\infty^{\ac}/K$ of $K = \QQ_{p^2}$. Theorem B implies that there is no such theory if in addition we require that $\phi_q(T) \in \OO_K \dcroc{T}$.

\subsection*{Acknowledgements}  
Thanks to K\^az{\i}m B\"uy\"ukboduk for asking me about universal norms in the anticyclotomic tower, which started this train of thought. Thanks to Jean-Marc Fontaine for some useful discussions on this topic, in particular for a remark that led to theorem B above. Thanks to Bryden Cais and Christopher Davis for their comments about the first version of this note, to Sandra Rozensztajn for remark \ref{sr}, and to the two referees for pointing out some inaccuracies and suggesting some improvements.

\section{Lifting the field of norms}\label{fonsec}

Let $K$ be a finite extension of $\Qp$ and let $K_\infty$ be an infinite and totally ramified Galois extension of $K$ that is ``strictly arithmetically profinite'' (see \S 1.2.1 of \cite{WCN} for the definition, which we don't use; \emph{arithmetically profinite} means that the ramification subgroups $\Gamma_K^u$ of $\Gamma_K$ are open and strictness is an additional condition). Note that if $\Gamma_K = \Gal(K_\infty/K)$ is a $p$-adic Lie group, then as recalled in \S 1.2.2 of \cite{WCN}, it follows from the main theorem of \cite{S72} that $K_\infty/K$ is strictly arithmetically profinite.

We can apply to $K_\infty/K$ the ``field of norms'' construction of \cite{FW,FW2} and \cite{WCN}, which we now recall. Let $\cF$ denote the set of finite extensions $F$ of $K$ that are contained in $K_\infty$, and let $X_K(K_\infty)$ denote the set of sequences $(x_F)_{F \in \cF}$ such that $\Nm_{F_2/F_1}(x_{F_2}) = x_{F_1}$ whenever $F_1 \subset F_2$. By the results of \S 2 of \cite{WCN}, one can endow $X_K(K_\infty)$ with the structure of a field, a field  embedding of $k_K =k_{K_\infty}$ in $X_K(K_\infty)$, and a valuation $\val(\cdot)$ (where $\val(x)$ is the common value of the $\val_F(x_F)$ for $F \in \cF$). We then have the following theorem (theorem 2.1.3 of \cite{WCN}).

\begin{theo}\label{cdnthm}
The field $X_K(K_\infty)$ is a complete valued field with residue field $k_K$.
\end{theo}

If $\pi_K$ denotes a uniformizer of $X_K(K_\infty)$, then $X_K(K_\infty) = k_K \dpar{\pi_K}$. The group $\Gamma_K$ acts on $X_K(K_\infty)$. If $q = \Card(k_K)$, then we have the $k_K$-linear Frobenius map $\phi_q : X_K(K_\infty) \to X_K(K_\infty)$ given by $x \mapsto x^q$, and it commutes with the action of $\Gamma_K$. 

Let $E$ be a finite extension of $\Qp$ such that $k_E = k_K$, let $\varpi_E$ be a uniformizer of $E$ and let $\afont_K$ denote the $\varpi_E$-adic completion of $\OO_E \dcroc{T}[1/T]$. The ring $\afont_K$ is a $\varpi_E$-Cohen ring for $X_K(K_\infty)$, that is a complete discrete valuation ring whose maximal ideal is generated by $\varpi_E$ and whose residue field is $X_K(K_\infty)$. The question that we want to ask is: when can we lift the action of $\Gamma_K$ on $X_K(K_\infty)$ to an $\OO_E$-linear action on $\afont_K$, along with a compatible $\OO_E$-linear Frobenius lift?

\begin{enonce}{Question}\label{quest}
Are there power series $\{F_g(T)\}_{g \in \Gamma_K}$ and $P(T)$ in $\afont_K$, such that
\begin{enumerate}
\item $\overline{F}_g(\pi_K) = g(\pi_K)$ and $\overline{P}(\pi_K) = \pi_K^q$ in $k_K \dpar{\pi_K}$,
\item $F_g \circ P = P \circ F_g$ and $F_h \circ F_g = F_{gh}$ whenever $g,h \in \Gamma_K$?
\end{enumerate}
If the answer to this question is ``yes'', then we say that the action of $\Gamma_K$ is \emph{liftable}.
\end{enonce}

If $K$ is unramified over $\Qp$ and $K_\infty = K(\mu_{p^\infty})$ is the cyclotomic extension, then the action of $\Gamma_K$ is liftable, since for the uniformizer $\pi_K = ( (\zeta_{p^n}-1)_{K(\zeta_{p^n})} )_{n \geq 1}$ we can take $P(T)=(1+T)^p-1$ and $F_g(T)=(1+T)^{\chi_{\cyc}(g)}-1$. More generally, if $K_\infty/K$ is the extension generated by the torsion points of a Lubin-Tate formal $\OO_F$-module, with $K/F$ unramified, then the action of $\Gamma_K$ is liftable with $P(T)=[\varpi_F](T)$ and $F_g(T)=[g](T)$ in appropriate coordinates.

Write $\efont_K$ for the field $X_K(K_\infty)$ (this notation is somewhat standard, but unfortunate considering the fact that $\efont_K$ depends on $K_\infty$ but not on $K$). Recall that every finite separable extension of $\efont_K$ is of the form $\efont_L$ where $L$ is a finite extension of $K$ (\S 3.2 of \cite{WCN}), and that to the extension $\efont_L/\efont_K$, there corresponds a unique \'etale extension of $\varpi_E$-rings $\afont_L / \afont_K$. Indeed, if $\efont_L = \efont_K[X] / Q(X)$, then we can take $\afont_L = \afont_K[X]/\widetilde{Q}(X)$, where $\widetilde{Q}$ is a unitary polynomial that lifts $Q$, and the resulting ring depends only on $Q(X)$ by Hensel's lemma.

\begin{theo}\label{life}
Let $L$ be a finite extension of $K$ and let $L_\infty=L K_\infty$. If the action of $\Gamma_K$ on $\efont_K$ is liftable, then the action of $\Gamma_L$ on $\efont_L$ is liftable. 
\end{theo}

\begin{proof}
Note that $\Gamma_L$ injects into $\Gamma_K$. Since $\afont_L / \afont_K$ is \'etale, the Frobenius map $\phi_q$ and the action of $g \in \Gamma_K$ extend to $\afont_L$ (for exemple, if $x \in \afont_L$ satisfies $Q(x)=0$ with $Q(X) \in \afont_K[X]$ unitary, then $(gQ)(g(x)) = 0$ has a solution in $\afont_L$ by Hensel's lemma). There exists an element $T_L \in \afont_L$ lifting $\pi_L$ such that $\afont_L$ is the $\varpi_E$-adic completion of $\OO_E \dcroc{T_L}[1/T_L]$. We can take $F_g(T_L) = g(T_L)$ and $P(T_L)=\phi_q(T_L)$. Note that if $L_\infty/K_\infty$ is not totally ramified, then we may need to replace $E$ by a larger unramified extension of degree $d$, and $\phi_q$ by $\phi_q^d$ accordingly.
\end{proof}

Even in the case of cyclotomic extensions, the series $F_g(T)$ and $P(T)$ can be quite complicated if $L/K$ is ramified. For example, suppose that $\pi_L=\pi_K^{1/n}$ with $p \nmid n$ (this corresponds to a tamely ramified extension $L/K$). We can then take $T_L = T_K^{1/n}$ and 
\[ \phi(T_L) = \left((1+T_K)^p-1\right)^{1/n} = T_L^p \cdot \left(1+\frac{p}{T_L^n} + \cdots + \frac{p}{T_L^{n(p-1)}} \right)^{1/n}, \]
so that $P(T_L)$ is overconvergent but does not belong to $\OO_E \dcroc{T_L}$.

\begin{theo}\label{liftdown}
Let $F_\infty \subset K_\infty$ be a Galois subextension such that $K_\infty/F_\infty$ is finite, and let $\Gamma_F= \Gal(F_\infty/K)$. If the action of $\Gamma_K$ on $\efont_K$ is liftable, then the action of $\Gamma_F$ on $\efont_F$ is liftable. 
\end{theo}

\begin{proof}
We check that the action of $\Gamma_F$ lifts to $\afont_F = \afont_K^{\Gal(K_\infty/F_\infty)}$. The ring $\afont_F$ is stable under $\Gamma_F$ and $\phi_q$ by construction, and its image in $\efont_K$ is $\efont_F$ since $\afont_F$ contains both $\OO_E$ and $\Nm_{K_\infty/F_\infty}(T)$.
\end{proof}

\section{Application to $(\phi,\Gamma)$-modules}\label{pgmsec}

One reason for asking question \ref{quest} is that it is relevant to the theory of $(\phi,\Gamma)$-modules for $\OO_E$-representations of $G_K$. This theory has been developed in \cite{F90} when $K_\infty = K(\mu_{p^\infty})$, but it can easily be generalized to other extensions $K_\infty/K$ for which the action of $\Gamma_K$ is liftable, as was observed for example in \S 2.1 of \cite{AS06}. For instance, the generalization to Lubin-Tate extensions is explicitely carried out in \S 1 of \cite{KR09} and is further discussed in \cite{FX13} and \cite{CE13}. Let $\afont$ be the $\varpi_E$-adic completion of $\varinjlim_L \afont_L$, where $L$ runs through the set of finite extensions of $K$. Let $H_K = \Gal(\Qpbar/K_\infty)$.

\begin{theo}\label{questpgm}
If the action of $\Gamma_K$ is liftable, then there is an equivalence of categories
\[ \text{$\{ (\phi_q,\Gamma_K)$-modules on $\afont_K \} \longleftrightarrow \{ \OO_E$-linear representations of $G_K\}$}, \] 
given by the mutually inverse functors $\dfont \mapsto (\afont \otimes_{\afont_K} \dfont)^{\phi_q=1}$ and $V \mapsto (\afont \otimes_{\OO_E} V)^{H_K}$.
\end{theo}

\begin{proof}
The proof follows \S A.1.2 and \S A.3.4 of \cite{F90} as well as \S 2.1 of \cite{AS06}, and we sketch it here. Note that $\afont^{\phi_q=1} = \OO_E$ since $q=\Card(k_E)$. Let $\efont = \efont_K^{\sep}$, so that $\afont / \varpi_E \afont = \efont$. 

The theory of $\phi$-modules tells us that if $M$ is a $\phi_q$-module over $\efont$, then $M = \efont \otimes_{k_E} M^{\phi_q=1}$ and that $1-\phi_q : M \to M$ is surjective. These two facts imply that if $D$ is a $\phi_q$-module over $\afont_K$, then $\afont \otimes_{\afont_K} D = \afont \otimes_{\OO_E} V(D)$ with $V(D) = (\afont \otimes_{\afont_K} D)^{\phi_q=1}$.

Conversely, Hilbert's theorem 90 says that $H^1(\Gal(\efont/\efont_K),\GL_d(\efont))$ is trivial for all $d \geq 1$. The theory of the field of norms gives us an isomorphism between $\Gal(\efont/\efont_K)$ and $H_K$ (\S 3.2 of \cite{WCN}). By d\'evissage, this implies that if $V$ is an $\OO_E$-representation of $H_K$, then $\afont \otimes_{\OO_E} V = \afont \otimes_{\afont_K} D(V)$ where $D(V) = (\afont \otimes_{\OO_E} V)^{H_K}$.

These two facts imply that the functors of the theorem are mutually inverse.
\end{proof}

\section{Embeddings into rings of periods}\label{atsec}

We now explain how to view the different rings whose construction we have recalled as subrings of some of Fontaine's rings of periods (constructed for example in \cite{FPP}). Let $I$ be the ideal of elements of $\OO_{\Cp}$ with valuation at least $1/p$. Let $\et$ denote the fraction field of $\etplus=\projlim_{x \mapsto x^p} \OO_{\Cp}/I$. Let $\et_K=\et^{H_K}$. By \S 4.2 of \cite{WCN}, there is a canonical $G_K$-equivariant embedding $X_K(K_\infty) \to \et_K$ and we also denote its image by $\efont_K$. 

Let $W_E(\cdot) = \OO_E \otimes_{\OO_{E_0}} W(\cdot)$ denote the $\varpi_E$-Witt vectors. Let $\at = W_E(\et)$, and endow it with the $\OO_E$-linear Frobenius map $\phi_q$ and the $\OO_E$-linear action of $G_K$ coming from those on $\et$. These are well-defined since $E_0 \subset K$. Let $\at_K = \at^{H_K}$ so that $\at_K = W_E(\et_K)$. If $\afont_K$ is equipped with a lift of the action of $\Gamma_K$ and a commuting Frobenius map $\phi_q$, then there is an embedding $\afont_K \to \at_K$ that is compatible with $\phi_q$, $\Gamma_K$-equivariant, and lifts the embedding $\efont_K \to \et_K$. See \S A.1.3 of \cite{F90} for a proof, or simply remark that since $\afont_K$ is the $\varpi_E$-adic completion of $\OO_E \dcroc{T}[1/T]$, it is enough to show that there exists one and only one element $v \in \at_K$ (the image of $T$) that lifts $\pi_K$ and satisfies $\phi_q(v) = P(v)$. This now follows from the fact that if $S$ denotes the set of elements of $\at_K$ whose image in $\et_K$ is $\pi_K$, then $x \mapsto \phi_q^{-1} (P(x))$ is a contracting map on $S$. Let $v \in \at_K$ be the image of $T$ as above, so that $\phi_q(v) = P(v)$ and $g(v) = F_g(v)$ for all $g \in \Gamma_K$. Note that $\overline{v} = \pi_K \in \etplus$. Let $\atplus = W_E(\etplus)$

\begin{lemm}\label{regvat}
If $P(T) \in \OO_E \dcroc{T}$, then $v \in \atplus$.
\end{lemm}

\begin{proof}
We have $v - [\overline{v}] \in \varpi_E  \at$ so that $v \in \atplus + \varpi_E \at$. Suppose that $v \in  \atplus + \varpi_E^k  \at$ for some $k \geq 1$. We have $P(T) \in T^q + \MM_E \dcroc{T}$. This implies that $P(v) \in  \atplus + \varpi_E^{k+1}  \at$ and hence also $v = \phi_q^{-1} (P(v)) \in  \atplus + \varpi_E^{k+1} \at$. By induction on $k$, we get $v \in \atplus$. 
\end{proof}

In \S \ref{finht}, we use the fact that if $L$ contains $K$ and $E$, then $\atplus$ injects into $\bdr^+$ in a $G_L$-equivariant way. We also use the following lemma about $\bdr^+$.

\begin{lemm}\label{cvthbdr}
Let $E$ be a finite extension of $\Qp$ and take $f(T) \in E \dcroc{T}$. If $x \in \bdr^+$, then the series $f(x)$ converges in $\bdr^+$ if and only if the series $f(\theta(x))$ converges in $\Cp$.
\end{lemm}

\begin{proof}
We prove that the series converges in $\bdr^+/t^k$ for all $k \geq 1$. Recall that $\bdr^+/t^k$ is a Banach space, the unit ball being the image of $\atplus \to \bdr^+ / t^k$. We can enlarge $E$ so that it contains an element of valuation $\vp(\theta(x))$ and it is then enough to prove that if $\theta(x) \in \OO_{\Cp}$, then $\{x^n\}_{n \geq 0}$ is bounded in $\bdr^+/t^k$. Let $\omega$ be a generator of $\ker(\theta : \atplus \to \OO_{\Cp})$ and let $x_0$ be an element of $\atplus$ such that $\theta(x) = \theta(x_0)$. We can write $x=x_0+\omega y + t^k z$ where $y \in \atplus[1/p]$ and $z \in \bdr^+$. We then have
\[ x^n = x_0^n + \binom{n}{1} x_0^{n-1} \omega y + \cdots + \binom{n}{k-1} x_0^{n-(k-1)} (\omega y)^{k-1} + t^k z_k, \]
with $z_k \in \bdr^+$, so that $x^n \in (\atplus + y \atplus + \cdots + y^{k-1} \atplus) + t^k \bdr^+$ for all $n$.
\end{proof}

\section{Lifts of finite height}\label{finht}

In this section we prove theorem B, which we now recall.

\begin{theo}\label{carphiq}
If the action of $\Gamma_K$ is liftable with $\phi_q(T) \in \OO_E \dcroc{T}$, then $\Gamma_K$ is abelian, and there is an injective character $\Gamma_K \to \OO_E^\times$, whose conjugates by $\Emb(E,\Qpbar)$ are all de Rham with weights in $\ZZ_{\geq 0}$.
\end{theo}

Before proving theorem \ref{carphiq}, we give a number of intermediate results to the effect that if $P(T) = \phi_q(T)$ belongs to $\OO_E \dcroc{T}$, then one can improve the regularity of the power series $P(T)$ and $F_g(T)$ for $g \in \Gamma_K$.

\begin{prop}\label{colpow}
If $P(T) \in \OO_E \dcroc{T}$, then $F_g(T) \in T \cdot \OO_E \dcroc{T}$ for all $g \in \Gamma_K$.
\end{prop}

\begin{proof}
The ring $\afont_K$ is a free $\phi_q(\afont_K)$-module of rank $q$. As in \S 2.3 of \cite{F90}, let $\Nphi : \afont_K \to \afont_K$ denote the map 
\[ \Nphi : f(T) \mapsto \phi_q^{-1} \circ \Nm_{\afont_K/\phi_q(\afont_K)} (f(T)). \]
If $P(T) \in \OO_E \dcroc{T}$, then $\Nphi(\OO_E \dcroc{T}) \subset \OO_E \dcroc{T}$ since the ring $\OO_E \dcroc{T}$ is a free $\OO_E \dcroc{P(T)}$-module of rank $q$. Furthermore, we have $\Nphi(1+\varpi_E^k \afont_K) \subset 1+\varpi_E^{k+1} \afont_K$ if $k \geq 1$ (see 2.3.2 of ibid). This implies that if $k \geq 1$, then  
\[ \Nphi(\OO_E \dcroc{T}^\times+\varpi_E^k \afont_K) \subset \OO_E \dcroc{T}^\times+\varpi_E^{k+1} \afont_K, \]
and likewise, since $\Nphi(T)=T$ and $T$ is invertible in $\afont_K$,
\[ \Nphi(T \cdot \OO_E \dcroc{T}^\times+\varpi_E^k \afont_K) \subset T \cdot \OO_E \dcroc{T}^\times+\varpi_E^{k+1} \afont_K. \]
This implies, by induction on $k$, that $(T \cdot \OO_E \dcroc{T}^\times+ \varpi_E \afont_K)^{\Nphi(x)=x} \subset T \cdot \OO_E \dcroc{T}^\times$. 

We have $F_g(T) \in T \cdot \OO_E \dcroc{T}^\times+\varpi_E \afont_K$ and since $\Nphi$ commutes with the action of $\Gamma_K$, we have $\Nphi(g(T))=g(T)$ and hence $F_g(T) \in (T \cdot \OO_E \dcroc{T}^\times+\varpi_E \afont_K)^{\Nphi(x)=x} \subset T \cdot \OO_E \dcroc{T}^\times$.
\end{proof}

\begin{rema}\label{colover}
The same proof implies that if $P(T)$ is overconvergent, then so is $F_g(T)$.
\end{rema}

\begin{lemm}\label{reduc1}
If $P(T) \in \OO_E \dcroc{T}$, then there exists $a \in \MM_E$ such that if $T'=T-a$, then $\phi_q(T')=Q(T')$ with $Q(T') \in T' \cdot \OO_E \dcroc{T'}$.
\end{lemm}

\begin{proof}
Let $R(T)=P(T+a)$. We have $\phi_q(T')=\phi_q(T-a)=P(T)-a=R(T')-a$ so it is enough to find $a \in \MM_E$ such that $P(a)=a$. The Newton polygon of $P(T)-T$ starts with a segment of length $1$ and slope $-\vp(P(0))$, which gives us such an $a$ with $\vp(a)=\vp(P(0))$.
\end{proof}

\begin{lemm}\label{reduc3}
If $P(T) \in T \cdot \OO_E \dcroc{T}$, then $P'(0) \neq 0$.
\end{lemm}

\begin{proof}
By proposition \ref{colpow}, we have $F_g(T)  \in T \cdot \OO_E \dcroc{T}$ for all $g \in \Gamma_K$. Write $F_g(T)=f_1(g) T + \bigO( T^2 )$ and $P(T) = \pi_k T^k + \bigO(T^{k+1})$ with $\pi_k \neq 0$. Note that $g \mapsto f_1(g)$ is a character $f_1 : \Gamma_K \to \OO_E^\times$. The fact that $F_g (P(T)) = P(F_g(T))$ implies that $f_1(g) \pi_k = \pi_k f_1(g)^k$ so that if $k \neq 1$, then $f_1(g)^{k-1}=1$. 

In particular, taking $g$ in the open subgroup $f_1^{-1}(1+2p \OO_E)$ of $\Gamma_K$, we must have $f_1(g)=1$. Take such a $g \in \Gamma_K \setminus \{ 1 \}$; since $\overline{F}_g(T) \neq T$, we can write $F_g(T) = T + T^i h(T)$ for some $i \geq 2$ with $h(0) \neq 0$. The equation $F_g (P(T)) = P(F_g(T))$ and the fact that $P(T + T^i h(T))=\sum_{j \geq 0} (T^i h(T))^j P^{(j)}(T)/j!$ imply that
\[ P(T) + P(T)^i h(P(T))  = P(T) + T^i h(T) P'(T) + \bigO (T^{2i+k-2}), \]
so that $P(T)^i h(P(T))  = T^i h(T) P'(T) + \bigO (T^{2i+k-2})$. The term of lowest degree of the LHS is of degree $ki$, while on the RHS it is of degree $i+k-1$. We therefore have $ki=i+k-1$, so that $(k-1)(i-1)=0$ and therefore $k=1$.
\end{proof}

\begin{proof}[Proof of theorem \ref{carphiq}] 
By the preceding results, if $P(T) \in \OO_E \dcroc{T}$, then we can make a change of variable so that $P(T) \in T \cdot \OO_E \dcroc{T}$ and  $F_g(T)  \in T \cdot \OO_E \dcroc{T}$ for all $g \in \Gamma_K$. Write $P(T) = \sum_{k \geq 1} \pi_k T^k$. By lemma \ref{reduc3}, we have $\pi_1 \neq 0$. If $A(T) = \sum_{k \geq 1} a_k T^k \in E \dcroc{T}$ with $a_1=1$, then the equation $A(P(T)) = \pi_1 \cdot A(T)$ is given by 
\[  P(T) + a_2 P(T)^2 + \cdots = \pi_1 \cdot ( T + a_2 T^2 + \cdots). \]
Looking at the coefficient of $T^k$ in the above equation, we get the equation 
\[   x_{k,1} a_1 + \cdots + x_{k,k-1} a_{k-1} = a_k(\pi_1-\pi_1^k), \] 
where $x_{k,i}$ is the coefficient of $T^k$ in $P(T)^i$ and hence belongs to $\OO_E$. This implies that the equation $A(P(T)) = \pi_1 \cdot A(T)$ has a unique solution in $E \dcroc{T}$, and that $a_k \in \pi_1^{1-k} \cdot \OO_E$. In particular, the power series $A(T)$ belongs to $\OO_E \dcroc{T/\pi_1}$ and so has a nonzero radius of convergence. If $g \in \Gamma_K$, then we have 
\[ A(F_g(P(T))) = A(P(F_g(T))) = \pi_1 \cdot A(F_g(T)). \]
This implies that if $B(T) = f_1(g)^{-1} \cdot A(F_g(T))$, then $b_1=1$ and $B(P(T)) = \pi_1 \cdot B(T)$, so that $B(T)=A(T)$ and hence $A(F_g(T)) = f_1(g) \cdot A(T)$ for all $g \in \Gamma_K$. The map $g \mapsto f_1(g)$ is therefore injective, since $f_1(g)=1$ implies that $F_g(T)=T$ so that $g=1$.

Recall that in \S \ref{atsec}, we have seen that there is a map $\afont_K \to \at$ that commutes with $\phi_q$ and the action of $G_K$. Let $v \in \at$ be the image of $T$. By lemma \ref{regvat}, $v \in  \atplus$. We have $\theta(v) \in \MM_{\Cp}$ and $\theta(\phi_q^m(v)) = \theta (P \circ \cdots \circ P (v))$ so that there exists $m_0 \geq 0$ such that $\theta(\phi_q^m(v))$ is in the domain of convergence of $A(T)$ if $m \geq m_0$. By lemma \ref{cvthbdr}, the series $A(\phi_q^m(v))$ converges in $(\bdr^+)^{H_L}$ where $L=KE$ and if $g \in G_L$, then we have 
\[ g( A(\phi_q^m(v))) = A(F_g(\phi_q^m(v))) = f_1(g) \cdot A(\phi_q^m(v)). \]
We now show that $A(\phi_q^m(v)) \neq 0$ for some $m \geq m_0$. If $\theta (\phi_q^m(v)) = 0$ for some $m$, then $\phi_q^m(v) \in \Fil^k \setminus \Fil^{k+1} \bdr^+$ for some $k \geq 1$, and then $A(\phi_q^m(v)) \in \Fil^k \setminus \Fil^{k+1} \bdr^+$ as well, so that $A(\phi_q^m(v)) \neq 0$. If $\theta (\phi_q^m(v)) \neq 0$ for all $m \geq m_0$, then the sequence $\{ \theta(\phi_q^m(v)) \}_{m \geq m_0}$ converges to zero in $\Cp$ and if $A(\phi_q^m(v)) = 0$ for all $m \geq m_0$, then $A(\theta(\phi_q^m(v))) = 0$ for all $m \geq m_0$ and this implies that $A(T)=0$ since $0$ would not be an isolated zero of $A(T)$. There is hence some $m \geq m_0$ such that $A(\phi_q^m(v)) \neq 0$, and the fact that $g( A(\phi_q^m(v))) = f_1(g) \cdot A(\phi_q^m(v))$ if $g \in G_L$ implies that the character $g \mapsto f_1(g)$ is de Rham and that its weight is in $\ZZ_{\geq 0}$.

The conjugates of $g \mapsto f_1(g)$ are treated in the same way. If $h \in \Emb(E,\Qpbar)$, then choose some $n(h) \in \ZZ$ such that $h = [x \mapsto x^p]^{n(h)}$ on $k_E$ so that $h=\phi^{n(h)}$ on $\OO_{E_0}$. Define an element $h(v) \in W_{h(E)}(\etplus)$ by the formula $h(e \otimes a) = h(e) \otimes \phi^{n(h)}(a)$. If $v \in W_E(\etplus)$ satisfies $\phi_q(v) = P(v)$ and $g(v) = F_g(v)$ for $g \in \Gamma_K$, then $\phi_q(h(v)) = P^h (h(v))$ and $g(h(v)) = F^h_g(h(v))$. The same reasoning as above now implies that the character $g \mapsto h(f_1(g))$ is de Rham and that its weight is in $\ZZ_{\geq 0}$.
\end{proof}

\begin{exem}\label{ezpdrp}
If $E=\Qp$, then theorem B implies that $K_\infty \subset \Qp^{\ab} \cdot L$ where $L$ is a finite extension of $K$. Indeed, every de Rham character $\eta : G_K \to \Zp^\times$ is of the form $\chi_{\cyc}^r \cdot \mu$ for some $r \in \ZZ$ and some potentially unramified character $\mu$ (see \S 3.9 of \cite{FST}).
\end{exem}

More generally, the condition that there is an injective character $\eta : \Gamma_K \to \OO_E^\times$, whose conjugates by $\Emb(E,\Qpbar)$ are all de Rham with weights in $\ZZ_{\geq 0}$, imposes some constraints on $K_\infty/K$. Here is a simple example (recall that $E$ is a finite extension of $\Qp$ such that $k_E=k_K$).

\begin{prop}\label{kprim}
If $K$ is a Galois extension of $\Qp$ of degree $d$, where $d$ is a prime number, and if $\eta : \Gamma_K \to \OO_E^\times$ is a de Rham character with weights in $\ZZ_{\geq 0}$, then the Lie algebra of the image of $\eta$ is either $\{0\}$, $\Qp$ or $K$.
\end{prop}

\begin{proof}
By local class field theory, $\Gamma_K$ can be realized as a quotient of $\OO_K^\times$ and $\eta$ can be seen as a character $\eta : \OO_K^\times \to \OO_E^\times$. This character is then the product of a finite order character by $x \mapsto \prod_{h \in  \Gal(K / \Qp)}  h(x)^{a_h}$ where $a_h$ is the weight of $\eta$ at the embedding $h$, so that $a_h \in \ZZ_{\geq 0}$. It is therefore enough to prove that if $f : K \to K$ is defined by $f = \sum_{h \in  \Gal(K / \Qp)} a_h \cdot h$ with $a_h \in \ZZ_{\geq 0}$, then the image of $f$ is either $\{0\}$, $\Qp$ or $K$.

Let $g$ be a generator of $\Gal(K / \Qp)$ and write $a_i$ for $a_{g^i}$ if $i \in \ZZ/d\ZZ$. If $\sum_i a_i g^i(x) = 0$ for some $x \in K^\times$, then $\sum_i a_{i+j} g^i(x) = 0$ for all $j \in \ZZ/d\ZZ$. This implies that the circulant matrix $( a_{i+j} )_{i,j}$ is singular. Its determinant is $\prod_{j=0}^{d-1} \sum_{i=0}^{d-1} \zeta_d^{ij} a_i$ where $\zeta_d$ is a primitive $d$-th root of $1$. Since $d$ is a prime number and $a_i \in \ZZ_{\geq 0}$ for all $i$, we can have $\sum_{i=0}^{d-1} \zeta_d^{ij} a_i = 0$ for some $j$ if and only if all the $a_i$ are equal to each other. In this case, $f$ is equal to $a_0 \cdot \mathrm{Tr}_{K/\Qp}(\cdot)$. Otherwise, $f : K \to K$ is bijective. This proves the proposition.
\end{proof}

\begin{coro}\label{anticyclex}
If $K = \QQ_{p^2}$ and $K_\infty$ is the anticylotomic extension of $K$ and $E$ is a totally ramified extension of $\QQ_{p^2}$, then it is not possible to find a lift for $\phi_q$ and $\Gamma_K$ such that $\phi_q(T) \in \OO_E \dcroc{T}$. 
\end{coro}

\begin{rema}\label{sr}
If $d$ is not a prime number, then the conclusion of proposition \ref{kprim} does not necessarily hold anymore. This is already the case if $\Gal(K/\Qp) = \ZZ/4\ZZ$.
\end{rema}

\begin{rema}\label{lubin}
There is some similarity between our methods for proving theorem B and the constructions of \cite{JL94}. For example, the power series $A(T)$ constructed in the proof of theorem B is denoted by $\mathbf{L}_f$ in \S 1 of ibid.\ and called the \emph{logarithm}. Theorem B is then consistent with the suggestion on page 341 of ibid.\ that ``for an invertible series to commute with a noninvertible series, there must be a formal group somehow in the background''. Indeed, the existence of a de Rham character $\Gamma_K \to \OO_E^\times$ with weights in $\ZZ_{\geq 0}$ indicates that the extension $K_\infty/K$ must in some sense ``come from geometry''.
\end{rema}

\providecommand{\bysame}{\leavevmode ---\ }
\providecommand{\og}{``}
\providecommand{\fg}{''}
\providecommand{\smfandname}{\&}
\providecommand{\smfedsname}{\'eds.}
\providecommand{\smfedname}{\'ed.}
\providecommand{\smfmastersthesisname}{M\'emoire}
\providecommand{\smfphdthesisname}{Th\`ese}

\end{document}